\documentclass[12pt,bezier]{article}
\usepackage{amssymb}
\usepackage{mathrsfs}
\usepackage{amsmath}
\usepackage{amsfonts,amsthm,amssymb}
\usepackage{amsfonts}
\usepackage{graphicx,graphics}
\usepackage{subfigure}
\usepackage[]{caption2}
\usepackage{color}
\usepackage{cite}
\usepackage{hyperref}
\usepackage{indentfirst}
\newtheorem{lem}{Lemma}
\newtheorem{thm}{Theorem}

\renewcommand{\thethm}{\Alph{thm}}

\begin{document}
\title{A degree sum condition for the existence of a quasi 5-contractible edge in a quasi 5-connected graph}

\author{Shuai Kou$^a$, \quad Chengfu Qin$^b$, \quad Weihua Yang$^a$\footnote{Corresponding author. E-mail: yangweihua@tyut.edu.cn, ywh222@163.com},\\
\quad Mingzu Zhang$^c$, \quad Shuang Zhao$^a$\\
\small $^a$ Department of Mathematics, Taiyuan University of Technology, Taiyuan 030024, China\\
\small $^b$ Department of Mathematics, Nanning Normal University, Nanning 530001, China\\
\small $^c$ College of Mathematics and System Sciences, Xinjiang University, Urumqi 830046, China
}
\date{}

\maketitle {\flushleft\bf Abstract:} {\small An edge of a quasi $k$-connected graph is said to be quasi $k$-contractible if the contraction of the edge results in a quasi $k$-connected graph. We show that every 5-connected graph contains a quasi 5-contractible edge. Furthermore, we prove that a quasi 5-connected graph possesses a quasi 5-contractible edge, if the degree sum of any two vertices with distance at most two is at least 9. This result strengthens a theorem proved by Kriesell when $k=4$ (M. Kriesell, A degree sum condition for the existence of a contractible edge in a $k$-connected graph, J. Combin. Theory Ser. B 82(2001)81-101).}


\section{Introduction}
In this paper, all graphs considered are finite, simple and undirected graphs, with undefined terms and notations following \cite{Bondy}. For a graph $G$, let $V(G)$ and $E(G)$ denote the set of vertices and the set of edges of $G$, respectively. For $S\subseteq V(G)$, let $G-S$ denote the graph obtained from $G$ by deleting the vertices of $S$ together with the edges incident with them. For $x\in V(G)$, $N_{G}(x)$ denotes the set of neighbors of $x$ in $G$. By $d_{G}(x)=|N_{G}(x)|$ we denote the degree of a vertex $x\in V(G)$. Let $\delta(G)$ denote the minimum degree of $G$. The \emph{distance} between two vertices $x$ and $y$ in $G$ is the length of a shortest $x, y$-path in $G$.
An edge $e=xy$ of $G$ is said to be \emph{contracted} if it is deleted and its ends are identified. The resulting graph is denoted by $G/e$, and the new vertex in $G/e$ is denoted by $\overline{xy}$. Note that, in the contraction, each resulting pair of double edges is replaced by a single edge. Let $k\geq 2$ be an integer and let $G$ be a non-complete $k$-connected graph. An edge of $G$ is said to be \emph{k-contractible} if its contraction results in a $k$-connected graph. A $k$-connected graph without a $k$-contractible edge is said to be a \emph{contraction critical k-connected graph}.

A \emph{cut} of a connected graph $G$ is a subset $V^{\prime}(G)$ of $V(G)$ such that $G-V^{\prime}(G)$ is disconnected. A \emph{k-cut} is a cut of $k$ elements. Suppose that $T$ is a $k$-cut of $G$. We say that $T$ is a \emph{nontrivial k-cut}, if the components of $G-T$ can be partitioned into subgraphs $G_{1}$ and $G_{2}$ such that $|V(G_{1})|\geq 2$ and $|V(G_{2})|\geq 2$. A ($k-1$)-connected graph is \emph{quasi k-connected} if it has no nontrivial ($k-1$)-cuts. Clearly, every $k$-connected graph is quasi $k$-connected. Let $G$ be a quasi $k$-connected graph. An edge $e$ of $G$ is said to be \emph{quasi k-contractible} if $G/e$ is still quasi $k$-connected. If $G$ does not have a quasi $k$-contractible edge, then $G$ is said to be a \emph{contraction critical quasi k-connected graph}.

Tutte's \cite{Tutte1961} famous wheel theorem implies that every 3-connected graph on more than four vertices contains an edge whose contraction yields a new 3-connected graph. Thomassen \cite{Thomassen} stated that for $k\geq4$, there are infinitely many contraction critical $k$-connected $k$-regular graphs. Egawa \cite{Egawa} proved the following theorem, which is a sufficient condition for the existence of a $k$-contractible edge in a $k$-connected graph.

\begin{thm}\cite{Egawa}\label{thmA}
If $G$ is a non-complete $k$-connected graph with $\delta(G)\geq\lfloor\frac{5k}{4}\rfloor$, then $G$ contains a $k$-contractible edge.
\end{thm}

Later, Kriesell \cite{Kriesell} proved the following stronger statement.

\begin{thm}\cite{Kriesell}\label{thmB}
Let $G$ be a non-complete $k$-connected graph, where $k\neq7$. If $d_{G}(x)+d_{G}(y)\geq 2\lfloor\frac{5k}{4}\rfloor-1$ hold for every pair $x$, $y$ of vertices of $G$ at distance one or two,
then $G$ contains a $k$-contractible edge.
\end{thm}

Moreover, he conjectured that the degree sum condition can be relaxed for only adjacent pairs. Su and Yuan \cite{Su} verified that the conjecture holds for $k\geq8$.

\begin{thm}\cite{Su}\label{thmC}
Let $G$ be a non-complete $k$-connected graph, where $k\geq8$. If $d_{G}(x)+d_{G}(y)\geq 2\lfloor\frac{5k}{4}\rfloor-1$ hold for every pair $x$, $y$ of adjacent vertices of $G$, then $G$ contains a $k$-contractible edge.
\end{thm}

We focus on quasi 5-connected graphs and obtain the following results. Theorem \ref{thm2} strengthens Theorem \ref{thmB} when $k=4$, showing that if $G$ satisfies the degree sum condition, then not only does $G$ remain 4-connectivity after the contraction of an edge, but it also retains quasi 5-connectivity.

\setcounter{thm}{0}
\renewcommand{\thethm}{\arabic{thm}}
\begin{thm}\label{thm1}
Let $G$ be a 5-connected graph. Then $G$ contains a quasi 5-contractible edge.
\end{thm}

\begin{thm}\label{thm2}
Let $G$ be a quasi 5-connected graph. If $d_{G}(x)+d_{G}(y)\geq9$ hold for every pair $x$, $y$ of vertices of $G$ at distance one or two, then $G$ contains a quasi 5-contractible edge.
\end{thm}


\section{Preliminaries}
In this section, we introduce some more definitions and preliminary lemmas.

Let $P_{n}$, $C_{n}$, $K_{1,n-1}$ and $K_{n}$ denote the path, the cycle, the star and the complete graph on $n$ vertices, respectively. For two graphs $G$ and $H$, let $G\cup H$ denote the union of $G$ and $H$, and let $G+H$ denote the join of $G$ and $H$. Moreover, for a positive integer $m$, let $mG$ stand for the union of $m$ copies of $G$.

For a graph $G$, let $E(x)$ denote the set of edges incident with $x\in V(G)$. Let $V_{k}(G)$ denote the set of vertices of degree $k$ in $G$. For $S\subseteq V(G)$, let $N_{G}(S)=\cup_{x\in S}N_{G}(x)-S$ and let $G[S]$ denote the subgraph induced by $S$. For $S, S^{\prime}\subseteq V(G)$ with $S\cap S^{\prime}=\emptyset$, we denote by $E_{G}(S, S^{\prime})$ the set of edges between $S$ and $S^{\prime}$. For $E^{\prime}\subseteq E(G)$, let $G[E^{\prime}]$ denote the subgraph induced by $E^{\prime}$.
Let $\kappa(G)$ denote its \emph{connectivity}, and let $\mathcal{T}(G):=\{S\subseteq V(G): G-S$ is disconnected and $|S|=\kappa(G)\}$ denote the set of its smallest cut sets. For $T\in\mathcal{T}(G)$, a \emph{$T$-fragment} is the union of the vertex sets of at least one but not of all components of $G-T$.
A \emph{fragment} is a $T$-fragment for some $T\in\mathcal{T}(G)$.

Let $G$ be a $k$-connected graph.
A fragment $A$ of $G$ is called a \emph{nontrivial fragment} if $|A|\geq2$ and $|V(G)-(A\cup N_{G}(A))|\geq2$.
Let $G$ be a quasi $k$-connected graph and let $E_{0}=\{e\in E(G):G/e$ is $(k-1)$-connected, but not quasi $k$-connected$\}$. Then for $xy\in E_{0}$, $G$ has a $k$-cut $T$ containing $xy$. Moreover, $G-T$ can be partitioned into two subgraphs, each containing at least two vertices. The vertex set of each such subgraph is called a \emph{quasi T-fragment} of $G$ or, briefly, a \emph{quasi fragment}.
A nontrivial (quasi) fragment with least cardinality is called a \emph{nontrivial (quasi) atom}.
For an edge $e$ of $G$, a nontrivial (quasi) fragment $A$ of $G$ is said to be a \emph{nontrivial (quasi) fragment with respect to e} if $V(e)\subseteq N_{G}(A)$.
For a set of edges $E^{\prime}\subseteq E(G)$, we say that $A$ is a \emph{nontrivial (quasi) fragment with respect to $E^{\prime}$} if $A$ is a nontrivial (quasi) fragment with respect to some $e\in E^{\prime}$.

To establish Theorem \ref{thm1}, we prove the following lemma.

\begin{lem}\label{nontrivial}
Let $G$ be both a 5-connected and a contraction critical quasi 5-connected graph. Let $A$ be a nontrivial fragment of $G$. If there exists a vertex $x\in N_{G}(A)$ such that $|N_{G}(x)\cap A|=1$, then $|A|=2$.
\end{lem}

\begin{proof}
By the connectivity condition of the graph $G$, for each edge $e$ of $G$, there exists a nontrivial fragment with respect to $e$.
Suppose that $|A|\geq3$. Let $S=N_{G}(A)$ and $\overline{A}=V(G)-(A\cup S)$. Let $\{x_{1}\}=N_{G}(x)\cap A$, and let $B$ be a nontrivial fragment with respect to $xx_{1}$. Let $T=N_{G}(B)$ and $\overline{B}=V(G)-(B\cup T)$. Then we see that $x\in S\cap T$ and $x_{1}\in A\cap T$.
If $A\cap B=\emptyset$ and $A\cap\overline{B}=\emptyset$, then $|A\cap T|=|A|\geq3$, and hence $|\overline{A}\cap T|\leq1$. Since $|\overline{A}|\geq2$, we have that either $\overline{A}\cap B\neq\emptyset$ or $\overline{A}\cap\overline{B}\neq\emptyset$. Without loss of generality, assume that $\overline{A}\cap B\neq\emptyset$. Then $|S\cap B|\geq3$, and thus $|S\cap\overline{B}|\leq1$. But then $\overline{A}\cap\overline{B}=\emptyset$, so $|\overline{B}|\leq1$, a contradiction.
Therefore, $A\cap B\neq\emptyset$ or $A\cap\overline{B}\neq\emptyset$.

Without loss of generality, assume that $A\cap B\neq\emptyset$. Since $\{x_{1}\}=N_{G}(x)\cap A$, we have $N_{G}(x)\cap(A\cap B)=\emptyset$, and hence $|(S\cap B)\cup(S\cap T)\cup(A\cap T)|\geq6$. It follows that $|S\cap B|>|\overline{A}\cap T|$ and $|A\cap T|>|S\cap\overline{B}|$. Moreover, the fact $|S|+|T|=10$ assures us that $|(\overline{A}\cap T)\cup(S\cap T)\cup(S\cap\overline{B})|\leq4$, which implies that $\overline{A}\cap\overline{B}=\emptyset$.
If $A\cap\overline{B}\neq\emptyset$, we have $\overline{A}\cap B=\emptyset$ and $|S\cap\overline{B}|>|\overline{A}\cap T|$ by symmetry. Since $|S|=5$, $|\overline{A}\cap T|<2$, and thus $|\overline{A}|<2$, a contradiction.
So $A\cap\overline{B}=\emptyset$. Since $|\overline{B}|\geq2$, we have $|S\cap\overline{B}|\geq2$. Then $|A\cap T|\geq3$, and thus $|\overline{A}\cap T|\leq1$. This implies that $|(S\cap B)\cup(S\cap T)\cup(\overline{A}\cap T)|\leq4$, and so $\overline{A}\cap B=\emptyset$. But then $|\overline{A}|\leq1$, a contradiction.
\end{proof}

The following two lemmas appear as Lemma 5 and 7 in \cite{Kou}.

\begin{lem}\label{degree}
Let $G$ be a quasi 5-connected graph. If $xy\in E(G)$ and $\delta(G/xy)\geq4$, then $G/xy$ is 4-connected.
\end{lem}

\begin{lem}\label{triangle}
Let $G$ be a quasi 5-connected graph on at least 8 vertices. If there is a vertex $x\in V_{4}(G)$ such that $N_{G}(x)=\{x_{1}, x_{2}, x_{3}, x_{4}\}$ and $G[\{x_{1}, x_{2}, x_{3}\}]\cong K_{3}$, then $xx_{4}$ is a quasi 5-contractible edge of $G$.
\end{lem}

The following three lemmas are also needed in the proof of Theorem \ref{thm2}.

\begin{lem}\cite{Martinov}\label{description}
A 4-connected graph is contraction critical if and only if it is 4-regular and each of its edges belongs to a triangle.
\end{lem}

\begin{lem}\label{neighbor}
Let $G$ be a contraction critical quasi 5-connected graph, such that $d_{G}(u)+d_{G}(v)\geq9$ for every pair $u$, $v$ of vertices of $G$ at distance one or two. Then $G$ does not contain a vertex $x\in V_{4}(G)$ such that $G[N_{G}(x)]\cong4K_{1}$.
\end{lem}

\begin{proof}
Suppose that $G$ contains a vertex $x\in V_{4}(G)$ such that $G[N_{G}(x)]\cong4K_{1}$. Let $N_{G}(x)=\{x_{1}, x_{2}, x_{3}, x_{4}\}$. By Lemma \ref{degree}, $G/xx_{i}$ is 4-connected for $i\in\{1, 2, 3, 4\}$. Let $A$ be a quasi atom with respect to $E(x)$. Let $S=N_{G}(A)$ and $\overline{A}=V(G)-(A\cup S)$. Without loss of generality, assume that $A$ is a quasi fragment with respect to $xx_{1}$ and $x_{2}\in A$. Let $B$ be a quasi fragment with respect to $xx_{2}$. Let $T=N_{G}(B)$ and $\overline{B}=V(G)-(B\cup T)$. Then $x\in S\cap T$ and $x_{2}\in A\cap T$. Let $X_{1}=(S\cap B)\cup(S\cap T)\cup(A\cap T)$, $X_{2}=(S\cap B)\cup(S\cap T)\cup(\overline{A}\cap T)$, $X_{3}=(\overline{A}\cap T)\cup(S\cap T)\cup(S\cap\overline{B})$ and $X_{4}=(A\cap T)\cup(S\cap T)\cup(S\cap\overline{B})$.

Suppose $A\cap B\neq\emptyset$. Then $|X_{1}|\geq5$ since $N_{G}(x)\cap N_{G}(x_{2})=\emptyset$. If $|X_{1}|=5$, then $|A\cap B|=1$; Otherwise, $A\cap B$ is a quasi fragment with respect to $xx_{2}$ and $|A\cap B|<|A|$, a contradiction. However, we see that the vertex in $A\cap B$ has degree four and has distance one or two from $x$, a contradiction. Thus, $|X_{1}|\geq6$, which implies that $|A\cap T|>|S\cap\overline{B}|$ and $|X_{3}|\leq4$. Then we have $\overline{A}\cap\overline{B}=\emptyset$, and thus $|\overline{B}|<|A|$, contradicting the choice of $A$. So $A\cap B=\emptyset$, and similarly, $A\cap\overline{B}=\emptyset$.
Since $|A|\geq2$, we have $|A\cap T|\geq2$.
If $|X_{2}|\leq4$, then $\overline{A}\cap B=\emptyset$ and $|S\cap B|<|A\cap T|$. This implies $|B|<|A|$, a contradiction. So $|X_{2}|\geq5$. Similarly, $|X_{3}|\geq5$. Then we see that $|A\cap T|=|S\cap B|=|S\cap\overline{B}|=|\overline{A}\cap T|=2$ and $|S\cap T|=1$. Without loss of generality, we assume that $x_{1}\in S\cap B$. Let $A\cap T=\{a, x_{2}\}$, $S\cap B=\{x_{1}, b\}$ and $S\cap\overline{B}=\{a_{1}, a_{2}\}$. Then we have that $N_{G}(a)=\{x_{2}, x_{1}, b, a_{1}, a_{2}\}$ and $N_{G}(x_{2})=\{a, b, x, a_{1}, a_{2}\}$.

Note that $\delta(G/ax_{1})=4$. Then $G/ax_{1}$ is 4-connected by Lemma \ref{degree}. Let $C$ be a quasi fragment with respect to $ax_{1}$, and let $R=N_{G}(C)$, $\overline{C}=V(G)-(C\cup R)$. Clearly, $A=A\cap R$. Similar to the discussion in the previous paragraph, we have that $|S\cap C|=|S\cap\overline{C}|=2$ and $|S\cap R|=1$. That is, $x_{1}\in S\cap R$ and $|\{x, b, a_{1}, a_{2}\}\cap C|=|\{x, b, a_{1}, a_{2}\}\cap\overline{C}|=2$.
Without loss of generality, assume that $x\in C$. Thus, $x\in T\cap C$, $x_{1}\in B\cap R$ and $\{a, x_{2}\}\subseteq T\cap R$.
By degree sum condition, we have $d_{G}(b)\geq5$.
If $b\in C$, then $|(T\cap C)\cup(T\cap R)\cup(B\cap R)|\geq5$. Furthermore, $|(\overline{B}\cap R)\cup(T\cap R)\cup(T\cap\overline{C})|\geq5$ since $\{a_{1}, a_{2}\}\subseteq\overline{B}\cap\overline{C}$. This implies $|(T\cap C)\cup(T\cap R)\cup(B\cap R)|=5$. Then since $bx\notin E(G)$ and $N_{G}(\{a, x_{2}\})\cap(B\cap C)=\{b\}$, $|B\cap C|=2$. However, the other vertex in $B\cap C$ has degree four and is a neighbor of $x$, leading to a contradiction.
If $b\in\overline{C}$, we can similarly obtain a contradiction.
\end{proof}

\begin{figure}
  \centering
  \includegraphics{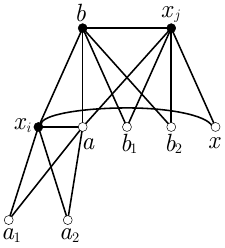}\\
  \caption{A black vertex indicates that the vertex has reached its maximum degree.}
  \label{fig1}
\end{figure}

\begin{lem}\label{subgraph}
Let $G$ be a quasi 5-connected graph, such that $d_{G}(u)+d_{G}(v)\geq9$ for every pair $u$, $v$ of vertices of $G$ at distance one or two. Suppose that Figure \ref{fig1} is a subgraph of $G$, where $d_{G}(x)=4$ and $xb_{1}, xb_{2}\notin E(G)$. Then $G/bx_{i}$ is quasi 5-connected.
\end{lem}

\begin{proof}
By Lemma \ref{degree}, $G/bx_{i}$ is 4-connected. Suppose that $G/bx_{i}$ is not quasi 5-connected. Then there exists a quasi fragment $A$ with respect to $bx_{i}$. Let $S=N_{G}(A)$ and $\overline{A}=V(G)-(A\cup S)$. Clearly, $x_{j}\in S$. If $|\{a, b_{1}, b_{2}, x\}\cap A|=1$, then $|N_{G}(\{b, x_{j}\})\cap A|=1$, which implies that $|A|=2$. However, another vertex in $A$ has degree four and has distance one or two from $x$, a contradiction. Therefore, $|\{a, b_{1}, b_{2}, x\}\cap A|\geq2$, and similarly, $|\{a, b_{1}, b_{2}, x\}\cap\overline{A}|\geq2$. Hence, $|\{a, b_{1}, b_{2}, x\}\cap A|=|\{a, b_{1}, b_{2}, x\}\cap\overline{A}|=2$.

Without loss of generality, assume that $x\in\overline{A}$.
If $a\in\overline{A}$, then $\{a_{1}, a_{2}\}\subseteq S\cup\overline{A}$, and thus $N_{G}(x_{i})\cap A=\emptyset$. It follows that $S-\{x_{i}\}$ is a nontrivial 4-cut of $G$, a contradiction. Therefore, $a\in A$. This implies that $\{a_{1}, a_{2}\}\subseteq A\cup S$. Without loss of generality, we assume that $b_{1}\in A$ and $b_{2}\in\overline{A}$. Then we see that $N_{G}(\{b, x_{i}, x_{j}\})\cap\overline{A}=\{b_{2}, x\}$. Thus, $|\overline{A}|\leq3$, for otherwise, $(S-\{b, x_{i}, x_{j}\})\cup\{b_{2}, x\}$ is a nontrivial 4-cut of $G$, which contradicts the quasi 5-connectivity of $G$.
If $|\overline{A}|=3$, we let $\{u\}=\overline{A}-\{b_{2}, x\}$. Then we see that $u\in N_{G}(x)$ and $d_{G}(u)=4$, a contradiction. So $\overline{A}=\{b_{2}, x\}$. Since $b_{2}x\notin E(G)$, $d_{G}(b_{2})=4$. However, two vertices $x$ and $b_{2}$ has distance two, a contradiction.
\end{proof}

\section{A proof of Theorem \ref{thm1}}
In this section, we give a proof of Theorem \ref{thm1}.

\begin{proof}[{\bf Proof of Theorem \ref{thm1}}]
Suppose that $G$ is a contraction critical quasi 5-connected graph. Then for each edge $e$ of $G$, there exists a nontrivial fragment with respect to $e$. Clearly, $|V(G)|\geq9$.
Take a nontrivial atom $A$. Let $S=N_{G}(A)$ and $\overline{A}=V(G)-(A\cup S)$.

\noindent{\bf Claim 1.} $|V(G)|\geq10$.

\begin{proof}
Suppose that $|V(G)|=9$. Then $|A|=|\overline{A}|=2$. Let $A=\{x, y\}$ and $S=\{v_{1}, v_{2}, v_{3}, v_{4}, v_{5}\}$. Without loss of generality, assume that $\{v_{1}, v_{2}, v_{3}, v_{4}\}\subseteq N_{G}(x)$ and $\{v_{1}, v_{2}, v_{3}, v_{5}\}\subseteq N_{G}(y)$. Let $B$ be a nontrivial fragment with respect to $xv_{4}$, and let $T=N_{G}(B)$, $\overline{B}=V(G)-(B\cup T)$. Clearly, $|B|=|\overline{B}|=2$ and $y\in T$.
If $|S\cap B|\leq1$, then $\overline{A}\cap B=\emptyset$, and thus $|B|\leq1$, a contradiction. Thus, $|B|=|S\cap B|=2$, and by symmetry, $|\overline{B}|=|S\cap\overline{B}|=2$. Without loss of generality, we assume that $S\cap B=\{v_{1}, v_{2}\}$ and $S\cap\overline{B}=\{v_{3}, v_{5}\}$. Then we see that none of the edges in $\{v_{1}v_{3}, v_{1}v_{5}, v_{2}v_{3}, v_{2}v_{5}\}$ exists.
Let $C$ be a nontrivial fragment with respect to $yv_{5}$, and let $R=N_{G}(C)$, $\overline{C}=V(G)-(C\cup R)$. Similarly, we can deduce that $x\in R$ and $|C|=|S\cap C|=|S\cap\overline{C}|=|\overline{C}|=2$. Without loss of generality, assume that $v_{1}\in C$. Since $v_{1}v_{5}\notin E(G)$, another vertex in $C$ is adjacent to both $v_{5}$ and $v_{1}$. This implies that $C=\{v_{1}, v_{4}\}$ and $\overline{C}=\{v_{2}, v_{3}\}$. Thus, none of the edges in $\{v_{1}v_{2}, v_{1}v_{3}, v_{4}v_{2}, v_{4}v_{3}\}$ exists. Then we see that $d_{G}(v_{2})\leq4$, a contradiction.
\end{proof}

Let $v_{1}\in S$ and $x\in N_{G}(v_{1})\cap A$. Let $B$ be a nontrivial fragment with respect to $xv_{1}$. Let $T=N_{G}(B)$ and $\overline{B}=V(G)-(B\cup T)$. We see that $x\in A\cap T$ and $v_{1}\in S\cap T$.

\noindent{\bf Claim 2.} If $A\cap B\neq\emptyset$, then $\overline{A}\cap\overline{B}\neq\emptyset$ and $|A\cap B|=1$;  If $A\cap\overline{B}\neq\emptyset$, then $\overline{A}\cap B\neq\emptyset$ and $|A\cap\overline{B}|=1$.

\begin{proof}
Suppose that $\overline{A}\cap\overline{B}=\emptyset$. Since $A\cap B\neq\emptyset$, then $|A\cap T|\geq|S\cap\overline{B}|$, and thus $|\overline{B}|<|A|$, a contradiction. So $\overline{A}\cap\overline{B}\neq\emptyset$. If $|A\cap B|\geq2$, then $|(S\cap B)\cup(S\cap T)\cup(A\cap T)|\geq6$; Otherwise, $A\cap B$ is a nontrivial fragment and $|A\cap B|<|A|$, which contradicts the choice of $A$. It follows that $|(\overline{A}\cap T)\cup(S\cap T)\cup(S\cap\overline{B})|\leq4$, so $\overline{A}\cap\overline{B}=\emptyset$, a contradiction. Thus $|A\cap B|=1$.
Another one can be proved similarly, we omit it.
\end{proof}

\noindent{\bf Claim 3.} $|A|\leq3$.

\begin{proof}
Suppose that $|A|\geq4$. If $A\cap B=\emptyset$ and $A\cap\overline{B}=\emptyset$, then $|A\cap T|=4$, $|S\cap T|=1$ and $\overline{A}\cap T=\emptyset$. Thus, we have that either $\overline{A}\cap B\neq\emptyset$ or $\overline{A}\cap\overline{B}\neq\emptyset$. Without loss of generality, assume that $\overline{A}\cap B\neq\emptyset$. This implies that $|S\cap B|=4$ and $S\cap\overline{B}=\emptyset$, and so $\overline{B}=\emptyset$, a contradiction.
Therefore, $A\cap B\neq\emptyset$ or $A\cap\overline{B}\neq\emptyset$.
Without loss of generality, we may assume that $A\cap B\neq\emptyset$. By Claim 2, $|A\cap B|=1$ and $\overline{A}\cap\overline{B}\neq\emptyset$. If $A\cap\overline{B}=\emptyset$, then $|A\cap T|\geq3$, and then, $|(S\cap T)\cup(\overline{A}\cap T)|\leq2$. This implies that $|S\cap\overline{B}|\geq3$, and so $|S\cap B|\leq1$. Thus, $\overline{A}\cap B=\emptyset$, and thus $|B|<|A|$, a contradiction.
So $A\cap\overline{B}\neq\emptyset$. By Claim 2, $|A\cap\overline{B}|=1$ and $\overline{A}\cap B\neq\emptyset$. Thus, $|A\cap T|\geq2$. Note that $|(S\cap B)\cup(S\cap T)\cup(\overline{A}\cap T)|\geq5$ and $|(\overline{A}\cap T)\cup(S\cap T)\cup(S\cap\overline{B})|\geq5$. It follows that $|S\cap B|\geq|A\cap T|$ and $|S\cap\overline{B}|\geq|A\cap T|$. Thus, $|S\cap B|=|S\cap\overline{B}|=|A\cap T|=|\overline{A}\cap T|=2$.

Let $A\cap B=\{a\}$, $A\cap T=\{b, x\}$ and $A\cap\overline{B}=\{c\}$. Clearly, $\{v_{1}a, v_{1}c\}\subseteq E(G)$.
Let $C$ be a nontrivial fragment with respect to $v_{1}a$, and let $R=N_{G}(C)$, $\overline{C}=V(G)-(C\cup R)$. Then $v_{1}\in S\cap R$ and $a\in A\cap R$.
From the arbitrariness of vertex $x$ in $A\cap N_{G}(v_{1})$, we can obtain that $|A\cap C|=|A\cap\overline{C}|=1$ and $|A\cap R|=2$. Since $ac\notin E(G)$, $c\in A\cap R$. Thus $v_{1}b\in E(G)$, and so $A\subseteq  N_{G}(v_{1})$. From the arbitrariness of vertex $v_{1}$ in $S$, $G[E_{G}(A, S)]$ is a complete bipartite graph, which is absurd.
\end{proof}

\noindent{\bf Claim 4.} $|A|=3$.

\begin{proof}
Suppose that $|A|=2$. Let $A=\{x, y\}$ and $S=\{v_{1}, v_{2}, v_{3}, v_{4}, v_{5}\}$.
Without loss of generality, we assume that $\{v_{1}, v_{2}, v_{3}, v_{4}\}\subseteq N_{G}(x)$ and $\{v_{1}, v_{2}, v_{3}, v_{5}\}\subseteq N_{G}(y)$. Let $B_{1}$ be a nontrivial fragment with respect to $xv_{4}$. Let $T_{1}=N_{G}(B_{1})$ and $\overline{B_{1}}=V(G)-(B_{1}\cup T_{1})$. Clearly, $y\in T_{1}$. Similar to Claim 1, we can obtain $|S\cap B_{1}|=|S\cap\overline{B_{1}}|=2$. Without loss of generality, assume that $S\cap B_{1}=\{v_{1}, v_{2}\}$ and $S\cap\overline{B_{1}}=\{v_{3}, v_{5}\}$.
Let $B_{2}$ be a nontrivial fragment with respect to $yv_{5}$. Let $T_{2}=N_{G}(B_{2})$ and $\overline{B_{2}}=V(G)-(B_{2}\cup T_{2})$. Then we have that $x\in T_{2}$ and $|S\cap B_{2}|=|S\cap\overline{B_{2}}|=2$. Note that $\{x, y\}\subseteq T_{1}\cap T_{2}$.

If $S\cap B_{2}=\{v_{1}, v_{2}\}$ and $S\cap\overline{B_{2}}=\{v_{3}, v_{4}\}$, then $\{v_{1}, v_{2}\}\subseteq B_{1}\cap B_{2}$, $v_{3}\in\overline{B_{1}}\cap\overline{B_{2}}$, $v_{4}\in T_{1}\cap\overline{B_{2}}$ and $v_{5}\in\overline{B_{1}}\cap T_{2}$. By the fact $|T_{1}|+|T_{2}|=10$, we have $|(T_{1}\cap B_{2})\cup(T_{1}\cap T_{2})\cup(B_{1}\cap T_{2})|=|(\overline{B_{1}}\cap T_{2})\cup(T_{1}\cap T_{2})\cup(T_{1}\cap\overline{B_{2}})|=5$. Then we see that $N_{G}(\{x, y\})\cap(\overline{B_{1}}\cap\overline{B_{2}})=\{v_{3}\}$, which implies $\overline{B_{1}}\cap\overline{B_{2}}=\{v_{3}\}$. Hence, $d_{G}(v_{3})=5$ and $N_{G}(v_{3})\supset\{x, y, v_{4}, v_{5}\}$. It follows that $|N_{G}(v_{3})\cap\overline{A}|=1$. Thus, $|\overline{A}|=2$ by Lemma \ref{nontrivial}, and thus, $|G|=9$, a contradiction.

Therefore, we have that either $S\cap B_{2}=\{v_{1}, v_{3}\}$ and $S\cap\overline{B_{2}}=\{v_{2}, v_{4}\}$ or $S\cap B_{2}=\{v_{1}, v_{4}\}$ and $S\cap\overline{B_{2}}=\{v_{2}, v_{3}\}$.
We may assume $S\cap B_{2}=\{v_{1}, v_{3}\}$ and $S\cap\overline{B_{2}}=\{v_{2}, v_{4}\}$ without loss of generality, and the discussion for the other case is similar. Then $v_{1}\in B_{1}\cap B_{2}$, $v_{2}\in B_{1}\cap\overline{B_{2}}$, $v_{3}\in\overline{B_{1}}\cap B_{2}$, $v_{4}\in T_{1}\cap\overline{B_{2}}$ and $v_{5}\in\overline{B_{1}}\cap T_{2}$. Thus, $|(T_{1}\cap B_{2})\cup(T_{1}\cap T_{2})\cup(B_{1}\cap T_{2})|\geq5$ and $|(B_{1}\cap T_{2})\cup(T_{1}\cap T_{2})\cup(T_{1}\cap\overline{B_{2}})|\geq5$. This implies that $|T_{1}\cap B_{2}|\geq|\overline{B_{1}}\cap T_{2}|$ and $|T_{1}\cap\overline{B_{2}}|\geq|\overline{B_{1}}\cap T_{2}|$. Since $|T_{1}\cap B_{2}|+|T_{1}\cap\overline{B_{2}}|\leq3$, $\overline{B_{1}}\cap T_{2}=\{v_{5}\}$. Since $v_{3}\in\overline{B_{1}}\cap B_{2}$, we have $|(T_{1}\cap B_{2})\cup(T_{1}\cap T_{2})\cup(\overline{B_{1}}\cap T_{2})|\geq5$, it follows that $T_{1}\cap\overline{B_{2}}=\{v_{4}\}$. We observe that $N_{G}(\{x, y\})\cap(\overline{B_{1}}\cap\overline{B_{2}})=\emptyset$. Then $\overline{B_{1}}\cap\overline{B_{2}}=\emptyset$, and thus $v_{4}v_{5}\in E(G)$. Moreover, $|N_{G}(v_{4})\cap\overline{B_{1}}|=|N_{G}(v_{5})\cap\overline{B_{2}}|=1$, which implies that $\overline{B_{1}}=\{v_{3}, v_{5}\}$ and $\overline{B_{2}}=\{v_{2}, v_{4}\}$ by Lemma \ref{nontrivial}. Then we see that $\{v_{3}v_{5}, v_{2}v_{4}\}\subseteq E(G)$, and so $G[\{v_{1}, v_{2}, v_{3}, v_{4}, v_{5}\}]\cong K_{1}\cup P_{4}$.

Let $B_{3}$ be a nontrivial fragment with respect to $xv_{1}$. Let $T_{3}=N_{G}(B_{3})$ and $\overline{B_{3}}=V(G)-(B_{3}\cup T_{3})$. If $y\in T_{3}$, then $A=A\cap T_{3}=\{x, y\}$, $v_{1}\in S\cap T_{3}$, and $|S\cap B_{3}|=|S\cap\overline{B_{3}}|=2$. This implies that $G[\{v_{2}, v_{3}, v_{4}, v_{5}\}]$ is not connected, a contradiction. So $y\notin T_{3}$. Without loss of generality, we assume that $y\in B_{3}$. Then we have $\{v_{2}, v_{3}, v_{5}\}\subseteq B_{3}\cup T_{3}$. Since $N_{G}(x)\cap\overline{B_{3}}\neq\emptyset$, $v_{4}\in\overline{B_{3}}$. It follows that $|N_{G}(x)\cap\overline{B_{3}}|=\{v_{4}\}$ and  $\{v_{2}, v_{5}\}\subseteq T_{3}$. By Lemma \ref{nontrivial}, $|\overline{B_{3}}|=2$. If $v_{3}\in T_{3}$, then $xy\in E(G)$ and $N_{G}(x)\cap B_{3}=\{y\}$, and thus, $|B_{3}|=2$ by Lemma \ref{nontrivial}. This implies that $|V(G)|=9$, a contradiction. So $v_{3}\in B_{3}$.

Let $B_{4}$ be a nontrivial fragment with respect to $yv_{1}$. Let $T_{4}=N_{G}(B_{4})$ and $\overline{B_{4}}=V(G)-(B_{4}\cup T_{4})$. Similar to the discussion on $B_{3}$, we have that $\{y, v_{1}, v_{3}, v_{4}\}\subseteq T_{4}$, $\{x, v_{2}\}\subseteq B_{4}$ and $v_{5}\in\overline{B_{4}}$ if we assume that $x\in B_{4}$ without loss of generality.
Clearly, $\{y, v_{3}\}\subseteq B_{3}\cap T_{4}$, $\{x, v_{2}\}\subseteq T_{3}\cap B_{4}$, $v_{1}\in T_{3}\cap T_{4}$, $v_{4}\in\overline{B_{3}}\cap T_{4}$ and $v_{5}\in T_{3}\cap\overline{B_{4}}$.
If $B_{3}\cap B_{4}\neq\emptyset$, then $|(T_{3}\cap B_{4})\cup(T_{3}\cap T_{4})\cup(B_{3}\cap T_{4})|\geq7$ since $N_{G}(\{x, y\})\cap(B_{3}\cap B_{4})=\emptyset$. Thus, $|T_{3}\cap T_{4}|=|\overline{B_{3}}\cap T_{4}|=|T_{3}\cap\overline{B_{4}}|=1$ and $|T_{3}\cap B_{4}|=|B_{3}\cap T_{4}|=3$. It follows that $B_{3}\cap\overline{B_{4}}=\overline{B_{3}}\cap\overline{B_{4}}=\emptyset$, and so $|\overline{B_{4}}|=1$, a contradiction.
Therefore, $B_{3}\cap B_{4}=\emptyset$.
Similarly, $B_{3}\cap\overline{B_{4}}=\overline{B_{3}}\cap B_{4}=\emptyset$. This implies that $N_{G}(x)\cap\overline{B_{3}}=\{v_{4}\}$ and $N_{G}(y)\cap\overline{B_{4}}=\{v_{5}\}$. By Lemma \ref{nontrivial}, $|\overline{B_{3}}|=|\overline{B_{4}}|=2$.
If $|\overline{B_{3}}\cap\overline{B_{4}}|=1$, then $|T_{3}\cap T_{4}|\geq3$, which is impossible. So $\overline{B_{3}}\cap\overline{B_{4}}=\emptyset$. It follows that $|V(G)|=9$, a contradiction.
\end{proof}

\noindent{\bf Claim 5.} $|A\cap T|=2$.

\begin{proof}
Suppose $|A\cap T|\neq2$. If $|A\cap T|=3$, then $A\cap B=A\cap\overline{B}=\emptyset$ and $|\overline{A}\cap T|\leq1$. It follows that either $\overline{A}\cap B\neq\emptyset$ or $\overline{A}\cap\overline{B}\neq\emptyset$. Without loss of generality, we assume that $\overline{A}\cap B\neq\emptyset$. Then $|S\cap B|\geq3$, and thus $|S\cap\overline{B}|\leq1$, so $\overline{A}\cap\overline{B}=\emptyset$. This implies that $|\overline{B}|\leq1$, a contradiction.
So $|A\cap T|=1$. By Claim 1, we have that $|A\cap B|=|A\cap\overline{B}|=1$, $\overline{A}\cap B\neq\emptyset$ and $\overline{A}\cap\overline{B}\neq\emptyset$. Thus, $|S\cap B|=|S\cap\overline{B}|=1$ and $|S\cap T|=3$.

Let $A\cap B=\{y\}$, $A\cap\overline{B}=\{z\}$, $S\cap B=\{v_{2}\}$ and $S\cap\overline{B}=\{v_{3}\}$.
If $v_{2}x\notin E(G)$, then $\{x, z\}$ is a nontrivial fragment, which contradicts the choice of $A$. So $v_{2}x\in E(G)$. Similarly, $v_{3}x\in E(G)$. Let $C$ be a nontrivial fragment with respect to $v_{2}x$, and let $R=N_{G}(C)$, $\overline{C}=V(G)-(C\cup R)$. Then $x\in A\cap R$ and $v_{2}\in S\cap R$.
Similar to the discussion of fragment $B$, we have $|A\cap R|\neq3$. We assume $A\cap C\neq\emptyset$ without loss of generality. Then similar to the proof of Claim 1, we can obtain that $|A\cap C|=1$ and $\overline{A}\cap\overline{C}\neq\emptyset$. If $|A\cap R|=2$, then we see that $yz\in E(G)$, a contradiction. So $|A\cap R|=1$. Hence, $|A\cap\overline{C}|=1$, and thus, $\overline{A}\cap C\neq\emptyset$. This implies that $zv_{2}\in E(G)$, a contradiction.
\end{proof}

By Claim 5, we have that either $|A\cap B|=1$ and $A\cap\overline{B}=\emptyset$ or $A\cap B=\emptyset$ and $|A\cap\overline{B}|=1$. Without loss of generality, we assume that the former holds. By Claim 1, $\overline{A}\cap\overline{B}=\emptyset$. Then $|S\cap B|=|S\cap\overline{B}|=2$ and $|S\cap T|=1$. Let $A\cap B=\{y\}$. Clearly, $N_{G}(y)=(S\cap B)\cup(S\cap T)\cup(A\cap T)$ and $v_{1}y\in E(G)$. Let $C$ be a nontrivial fragment with respect to $v_{1}y$, and let $R=N_{G}(C)$, $\overline{C}=V(G)-(C\cup R)$. From the arbitrariness of vertex $x$ in $A$, we have $|A\cap R|=2$. Then we see that $|N_{G}(y)\cap R|\geq2$.
Since $d_{G}(y)=5$, $|N_{G}(y)\cap C|=1$ or $|N_{G}(y)\cap\overline{C}|=1$. By Lemma \ref{nontrivial}, $|C|=2$ or $|\overline{C}|=2$, which contradicts the choice of $A$.
This completes the proof of Theorem \ref{thm1}.
\end{proof}

\section{A proof of Theorem \ref{thm2}}
In this section, we give a proof Theorem \ref{thm2}.

\begin{figure}
  \centering
  \includegraphics{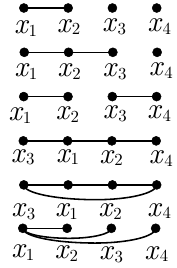}\\
  \caption{$K_{2}\cup 2K_{1}$, $P_{3}\cup K_{1}$, $2K_{2}$, $P_{4}$, $C_{4}$, $K_{1,3}$}
  \label{fig2}
\end{figure}

\begin{proof}[{\bf Proof of Theorem \ref{thm2}}]
Suppose that $G$ is a contraction critical quasi 5-connected graph. By Theorem \ref{thm1}, $\kappa(G)=4$.
If $|V(G)|\leq8$, then $G$ is also a contraction critical 4-connected graph. By Lemma \ref{description}, we have that $G$ does not satisfy the degree sum condition.
Therefore, $|V(G)|\geq9$.
Note that $\delta(G)=4$. Let $x\in V_{4}(G)$ with $N_{G}(x)=\{x_{1}, x_{2}, x_{3}, x_{4}\}$. By Lemmas \ref{triangle} and \ref{neighbor}, we have that $G[N_{G}(x)]$ is isomorphic to one of the six types: $K_{2}\cup 2K_{1}$, $P_{3}\cup K_{1}$, $2K_{2}$, $P_{4}$, $C_{4}$, $K_{1,3}$. For each type, the distribution of the four vertices is shown in Figure \ref{fig2}. Note that $x_{1}x_{2}\in E(G)$ holds in all types. Moreover, we have that both $G/xx_{3}$ and $G/xx_{4}$ are 4-connected by the degree sum condition and Lemma \ref{degree}.
Let $A$ and $B$ be quasi fragments with respect to $xx_{3}$ and $xx_{4}$, respectively. Let $S=N_{G}(A)$, $\overline{A}=V(G)-(A\cup S)$, $T=N_{G}(B)$ and $\overline{B}=V(G)-(B\cup T)$.  Without loss of generality, we assume that $x_{3}\in S\cap\overline{B}$ and $x_{4}\in\overline{A}\cap T$. Then we see that $\{x_{1}, x_{2}\}\subseteq V(G)-\overline{A}-\overline{B}$.
Let $X_{1}=(S\cap B)\cup(S\cap T)\cup(A\cap T)$, $X_{2}=(S\cap B)\cup(S\cap T)\cup(\overline{A}\cap T)$, $X_{3}=(\overline{A}\cap T)\cup(S\cap T)\cup(S\cap\overline{B})$ and $X_{4}=(A\cap T)\cup(S\cap T)\cup(S\cap\overline{B})$.

\noindent{\bf Claim 1.} If $A\cap B\neq\emptyset$, then $|X_{1}|\geq5$;
If $\overline{A}\cap B\neq\emptyset$, then $|X_{2}|\geq6$;
If $\overline{A}\cap\overline{B}\neq\emptyset$, then $|X_{3}|\geq6$;
If $A\cap\overline{B}\neq\emptyset$, then $|X_{4}|\geq6$.

\begin{proof}
If $A\cap B\neq\emptyset$ and $|X_{1}|=4$, then $|A\cap B|=1$. However, the vertex in $A\cap B$ has degree four, and has distance one from $x$, a contradiction. So $|X_{1}|\geq5$.
If $A\cap\overline{B}\neq\emptyset$, then $|X_{4}|\geq5$ since $N_{G}(x)\cap(A\cap\overline{B})=\emptyset$. If $|X_{4}|=5$, then $|A\cap\overline{B}|=1$. However, the vertex in $A\cap\overline{B}$ has degree four, and has distance two from $x$, a contradiction. So $|X_{4}|\geq6$.
The other two conclusions can be similarly obtained.
\end{proof}

\noindent{\bf Claim 2.} $|A\cap T|=|\overline{A}\cap T|=|S\cap B|=|S\cap\overline{B}|=2$, $|S\cap T|=1$ and $A\cap\overline{B}=\overline{A}\cap B=\overline{A}\cap\overline{B}=\emptyset$.

\begin{proof}
If $\overline{A}\cap T=\{x_{4}\}$, then $|X_{2}|\leq5$, and then, $\overline{A}\cap B=\emptyset$ by Claim 1. Since $|\overline{A}|\geq2$, $\overline{A}\cap\overline{B}\neq\emptyset$. By Claim 1, we have $|X_{3}|\geq6$. Then we see that $S\cap B=\emptyset$, and thus $N_{G}(x_{4})\cap B=\emptyset$. This implies that $T-\{x_{4}\}$ is a nontrivial 4-cut of $G$, a contradiction. So $|\overline{A}\cap T|\geq2$. Similarly, we have $|S\cap\overline{B}|\geq2$.
If $|A\cap T|\leq1$, then $|X_{1}|\leq4$, which implies that $A\cap B=\emptyset$ by Claim 1. Since $|A|\geq2$, $A\cap\overline{B}\neq\emptyset$. By Claim 1, $|X_{4}|\geq6$. It follows that $|A\cap T|=1$ and $S\cap B=\emptyset$, so $N_{G}(x)\cap B=\emptyset$, a contradiction. So $|A\cap T|\geq2$. Similarly, we have $|S\cap B|\geq2$. Therefore, $|A\cap T|=|\overline{A}\cap T|=|S\cap B|=|S\cap\overline{B}|=2$ and $|S\cap T|=1$.
Now, we see that $|X_{1}|=|X_{2}|=|X_{3}|=|X_{4}|=5$. By Claim 1, $A\cap\overline{B}=\overline{A}\cap B=\overline{A}\cap\overline{B}=\emptyset$.
\end{proof}

Let $S\cap\overline{B}=\{a, x_{3}\}$ and $\overline{A}\cap T=\{b, x_{4}\}$. Then $N_{G}(a)=\{x_{3}\}\cup(T-\{x\})$ and $N_{G}(b)=\{x_{4}\}\cup(S-\{x\})$.

\noindent{\bf Claim 3.} $G[N_{G}(x)]\cong 2K_{2}$.

\begin{proof}
By Claim 2, $N_{G}(x)\ncong K_{1,3}$; Otherwise, we find that $x_{1}\in S\cap T$, and thus, $|S\cap T|\geq2$, a contradiction.
If $N_{G}(x)\cong C_{4}$ or $N_{G}(x)\cong P_{4}$, then $x_{1}\in A\cap T$ and $x_{2}\in S\cap B$. Thus, $A\cap B=\emptyset$, for otherwise, $|A\cap B|=1$ and the vertex has degree four and has distance two from $x$, a contradiction. From this, it is not difficult to determine that edge $x_{2}u$ is quasi 5-contractible, where $\{u\}=(A\cap T)-\{x_{1}\}$. This contradicts the assumption that $G$ is a contraction critical quasi 5-connected graph.
If $N_{G}(x)\cong K_{2}\cup 2K_{1}$ or $N_{G}(x)\cong P_{3}\cup K_{1}$, then $x_{3}x_{4}\notin E(G)$. Thus, $N_{G}(x_{3})=\{a\}\cup(T-\{x_{4}\})$ and $N_{G}(x_{4})=\{b\}\cup(S-\{x_{3}\})$. It follows that $G$ has a subgraph isomorphic to Figure \ref{fig1}. By Lemma \ref{subgraph}, $bx_{3}$ is a quasi 5-contractible edge, a contradiction.
\end{proof}

By Claim 3 and the arbitrariness of $x\in V_{4}(G)$, we have that every vertex of degree four in $G$ such that its neighbor set is isomorphic to $2K_{2}$. Let $A\cap T=\{a_{1}, a_{2}\}$. Since $d_{G}(x_{3})\geq5$, we may assume $a_{1}x_{3}\in E(G)$ without loss of generality. Since two vertices $a_{1}$ and $x$ has distance two, $d_{G}(a_{1})\geq5$. Thus $\delta(G/aa_{2})=4$, which implies that $G/aa_{2}$ is 4-connected by Lemma \ref{degree}. Let $D$ be a quasi fragment with respect to $aa_{2}$, and let $Q=N_{G}(D)$, $\overline{D}=V(G)-(D\cup Q)$. Clearly, $x_{3}\in Q$.
Without loss of generality, assume that $\{b, x_{4}\}\subseteq D\cup Q$ and $a_{1}\in\overline{D}$.
If $x\in\overline{D}$, then $x_{4}\in Q$ and $b\in D$, and then $N_{G}(\{a, x_{3}\})\cap D=\{b\}$. It follows $|D|=2$. Let $D=\{b, v\}$. Then we have $d_{G}(v)=4$, and the distance between $v$ and $x$ is two, a contradiction. So $x\in D\cup Q$. Then we see that $N_{G}(\{a, x_{3}\})\cap\overline{D}=\{a_{1}\}$, which implies $|\overline{D}|=2$. Let $\overline{D}=\{a_{1}, u\}$. Then $d_{G}(u)=4$, and thus, $N_{G}(u)\cong 2K_{2}$. It follows $d_{G}(a_{1})=4$. However, the distance between $a_{1}$ and $x$ is two, a contradiction.
\end{proof}


\begin{thebibliography}{99}
\bibitem{Bondy} J.A. Bondy, U.S.R. Murty, Graph theory with application, The Macmillan Press Ltd, New York, 1976.
\bibitem{Egawa} Y. Egawa, Contractible edges in $n$-connected graphs with minimum degree greater than or equal to $[\frac{5n}{4}]$, Graphs Combin. 7(1991)15-21.
\bibitem{Kou} S. Kou, W. Yang, Contractible subgraphs of quasi 5-connected graphs, Discuss. Math. Graph Theory (2025)https://doi.org/10.7151/dmgt.2562.
\bibitem{Kriesell} M. Kriesell, A degree sum condition for the existence of a contractible edge in a $k$-connected graph, J. Combin. Theory Ser. B 82(2001)81-101.
\bibitem{Martinov} N. Martinov, A recursive characterization of the 4-connected graphs, Discrete Math. 84(1990)105-108.
\bibitem{Su} J. Su, X. Yuan, A new degree sum condition for the existence of a contractible edge in a $k$-connected graph, J. Combin. Theory Ser. B 96(2006)276-295.
\bibitem{Thomassen} C. Thomassen, Nonseparating cycles in $k$-connected graphs, J. Graph Theory 5(1981)351-354.
\bibitem{Tutte1961} W.T. Tutte, A theory of $3$-connected graphs, Neder. Akad. Wet. Proc. Ser. A 64(1961)441-455.
\end{thebibliography}
\end{document}